\documentclass[12pt]{article}

\usepackage{amssymb, amsmath, amsthm, mathtools, comment}

\newcommand{\E}{\mathbf{E}}
\newcommand{\p}{\mathbf{P}}
\newcommand{\R}{\mathbb{R}}
\newcommand{\N}{\mathbb{N}}
\newcommand{\fl}{\overline{f}}

\newcommand{\ind}{\mathbb{I}}

\theoremstyle{plain}
\newtheorem{theorem}{Theorem}
\newtheorem{lemma}[theorem]{Lemma}
\newtheorem{proposition}[theorem]{Proposition}
\newtheorem{corollary}[theorem]{Corollary}

\theoremstyle{remark}

\title{Functional limit theorem for branching processes \\ in 
nearly degenerate varying environment}
\author{P\'eter Kevei\footnote{Bolyai Institute, University of Szeged,
Hungary.
E-mail: \texttt{kevei@math.u-szeged.hu}} \and 
Kata Kubatovics\footnote{Bolyai Institute, University of Szeged, Hungary.
E-mail: \texttt{kubatovics@server.math.u-szeged.hu}}}

\date{}

\begin{document}

\maketitle

\begin{abstract}
We investigate branching processes in nearly degenerate varying environment, 
where the offspring distribution converges to the degenerate 
distribution at 1. Such processes die out almost surely, therefore, 
we condition on non-extinction or add inhomogeneous immigration.
Extending our one-dimensional limit results in \cite{KevKub}
we derive functional limit theorems.
In the former case, the limiting process is a 
time-changed simple birth-and-death process on $(-\infty, 0]$
conditioned on survival at $0$, while in the latter, it is a 
time-changed stationary continuous time branching process with 
immigration.

\end{abstract}

\section{Introduction}

A \emph{branching process in varying environment} (BPVE) $(X_n)_{n\in \N}$ is defined as
\begin{equation} \label{eq:X}
X_0 = 1, \quad \text{and } 
X_n = \sum_{j=1}^{X_{n-1}} \xi_{n,j}, \quad n\in\N = \{ 1, 2, \ldots \},
\end{equation}
where $\{\xi_{n,j}\}_{n,j\in\N}$ are nonnegative independent random variables such that for each $n$, $\{\xi_{n,j}\}_{j\in\N}$ are identically distributed; let $\xi_n$ denote a generic copy.
Adding immigration leads to a \emph{branching process in varying environment with immigration} (BPVEI) $(Y_n)_{n\in \N}$ defined as
\begin{equation} \label{eq:Y}
Y_0=0, \quad \text{and }
Y_n=\sum_{j=1}^{Y_{n-1}}\xi_{n,j}+\varepsilon_n, \quad n\in\N,
\end{equation}
where $\{\xi_n, \xi_{n,j}, \varepsilon_n\}_{n,j\in\N}$ are non-negative, independent random variables such that $\{\xi_n, \xi_{n,j}\}_{j\in\N}$ are identically distributed.

Put $\fl_n = \E (\xi_n)$ for the offspring mean in generation $n$. To exclude trivialities, we always assume that $\fl_n > 0$ for all $n$.
Let $f_n(s) = \E (s^{\xi_n})$, $s\in[0,1]$, be the generating function  (g.f.) of the offspring distribution in the $n$th generation.
As in \cite{KevKub}, we are interested in 
\emph{branching processes in nearly degenerate environment}, 
that is, we always assume the following conditions:
\begin{itemize}
\item[(C1)] $\lim_{n\to\infty}\fl_n = 1$, 
$\sum_{n=1}^{\infty}(1-\fl_n)_+ = \infty$, 
$\sum_{n=1}^\infty (\fl_n - 1)_+ < \infty$,
\item[(C2)] 
$\lim_{n\to\infty, \fl_n < 1} \frac{f_n''(1)}{1-\fl_n} = \nu \in [0,\infty)$,
and $\left( \frac{f_n''(1)}{|1- \fl_n|} \right)_{n}$ is bounded,
\item[(C3)] if $\nu>0$, then $\lim_{n\to \infty, \fl_n < 1}
\frac{f_n'''(1)}{1- \fl_n} = 0$, 
and $\left( \frac{f_n'''(1)}{|1- \fl_n|} \right)_{n}$ is bounded,
\end{itemize}
where here and later on $\lim_{n\to\infty, \fl_n < 1}$ means that the convergence holds along the subsequence $\{n: \fl_n <1\}$.
Since by (C1) subcritical regimes dominate in the process, 
$\E (X_n) = \prod_{j=1}^n \fl_j =: \fl_{0,n} \to 0$ as $n\to\infty$, 
the process dies out almost surely.
Condition (C2) implies that $f_n''(1) \to 0$, thus $f_n(s) \to s$, the branching mechanism converges to the degenerate branching.

\smallskip

Due to the recent work of Kersting \cite{Kersting} and the monograph by Kersting and Vatutin \cite{KerVat} BPVEs and BPVEIs have become an
interesting and current topic in probability; see the references 
in \cite{KevKub}.
Nearly degenerate BPVEIs were introduced in \cite{GyIPV} by Gy\"orfi et al., where it was assumed that the offspring have Bernoulli distribution, for references see \cite{KevKub}.

\smallskip

In this paper we investigate the long-term behavior of the paths
of nearly degenerate BPVEs and BPVEIs. 
To be able to derive convergence, it is crucial to find a proper scaling for the process. Condition (C1) implies that for all nearly degenerate BPVEs there is a sequence 
of natural numbers $(A(n))_{n\in\N}$ increasing to infinity such that
\begin{equation}\label{eq:A(n)}
\fl_{0,A(n)} \sim n^{-1}, \quad \text{as $n\to\infty$},
\end{equation}
which turns out to be the proper scaling.
In what follows, we are interested in the processes $(X_{A(nt)})_{t>0}$,
$(Y_{A(nt)})_{t > 0}$, and to ease notation $A(nt)$ is meant as $A(\lfloor nt \rfloor)$.

The simplest example is given by
\[
\fl_n = 
\begin{cases}
1 - \frac{\alpha}{n}, & n \geq \lfloor \alpha \rfloor + 1, \\
1, &  \text{otherwise},   
\end{cases}
\]
for some $\alpha > 0$. Then $\fl_{0,n} \sim c n^{-\alpha}$,
for some $c > 0$ as $n\to\infty$. Thus with 
$A(n) = \lfloor c^{1/\alpha} n^{1 / \alpha} \rfloor$, 
the convergence in \eqref{eq:A(n)} holds.

In Theorem \ref{thm:fdd>1}, we prove that $((X_{A(nt)})_{t\geq \varepsilon} | X_{A(n)} > 0)$ converges in distribution in the 
Skorokhod space of c\`adl\`ag functions 
to a time-transformed simple birth-and-death process conditioned 
on survival. We analyze the limiting process, and explore 
its relation to the quasi-stationary distribution.
In Theorem \ref{thm:conv-immig} we show that $(Y_{A(nt)})_{t > 0}$ converges 
to a time-transformed continuous time branching process with immigration. 
In both cases we only scale in time, therefore the 
limiting processes are 
continuous time Markov chains with state space $\N$, and 
$\N \cup \{ 0 \}$, respectively. Convergence to the 
Poisson process is a well-studied area; see 
\cite[Section 3.12]{Billingsley}. However, we are not aware 
of results on convergence to non-monotone Markov chains.
The proof of finite finite dimensional distributions is 
standard, but the proof of tightness is rather troublesome.

Section \ref{sec:main} contains the main results, while proofs are gathered in Section \ref{sec:proofs}.

\section{Main results}\label{sec:main}

\subsection{Nearly degenerate BPVEs}\label{subsec:BPVE}

Consider a simple continuous time birth-and-death process $(Z(t))_{t\geq -w}$
defined on the time interval $t \in [-w, \infty)$ for some $w \geq 0$, 
with birth rate $\tfrac{\nu}{2}$ and death rate $1 + \tfrac{\nu}{2}$.
The transition generating function of $Z$ is given as (see 
e.g.~\cite[Chapter III.5]{AthreyaNey})
\begin{equation} \label{eq:db-genfnc}
\begin{split}
F(s,t) & = 
\E [ s^{Z(t + u)} | Z(u ) = 1 ] \\
& = \frac{
\left( 1 + \frac{\nu}{2} \right) (s-1) - e^{t} (\frac{\nu}{2} s - (1 + \frac{\nu}{2}))
}{
\frac{\nu}{2} (s-1) - e^{t} (\frac{\nu}{2} s -(1 + \frac{\nu}{2}))} \\
& = 1 - e^{-t} \left( \frac{1}{1-s} + \frac{\nu}{2} ( 1 - e^{-t} ) \right)^{-1}.
\end{split}
\end{equation}

For $a < b \leq \infty$, 
the Skorokhod space of c\`adl\`ag functions on $[a,b]$  (or $[a, \infty)$)
with the usual metric is denoted by $D([a,b])$ (or $D([a,\infty))$),
and $\stackrel{\mathcal{D}}{\longrightarrow}$ stands for 
convergence in distribution in the Skorokhod space.
A random variable $Y$ has \emph{geometric distribution with parameter $p \in (0,1]$},
$Y \sim \text{Geom}(p)$, if $\p ( Y = k) = (1- p)^{k-1} p$, for $k = 1, 2, \ldots$.

\begin{theorem} \label{thm:fdd>1}
Let $(X_n)_{n\in\N}$ be a BPVE satisfying conditions (C1)--(C3). 
Then, for any $\varepsilon \in (0,1]$,
\[
\mathcal{L} ( (X_{A(nt)} )_{t \geq \varepsilon} | X_{A(n)} > 0 ) 
\stackrel{\mathcal{D}}{\longrightarrow} 
\mathcal{L} (({Z}(\log t))_{t \geq \varepsilon} | Z(0) > 0),
\]
as $n \to \infty$,
where $(Z(u))_{u\geq \log \varepsilon}$ is a simple birth-and-death process with initial distribution
$Z( \log \varepsilon ) \sim \text{Geom}(\frac{2}{2+\nu})$, birth rate 
$\frac{\nu}{2}$ and death rate $1 + \frac{\nu}{2}$.
\end{theorem}

Note that if $\nu = 0$, $Z(u)$ is a simple death process with death rate $1$ and initial distribution $Z( \log \varepsilon ) = 1$.

Since the limiting process $(Z(\log t))_{t \geq \varepsilon}$ 
is conditioned on non-extinction in $t=1$, 
it is no surprise that the limit behaves differently for $t\in[\varepsilon, 1)$ and $t\geq 1$. In fact, for $t\geq 1$, the limit reduces to $(Z(\log t))_{t\geq 1}$ with 
initial distribution $\textrm{Geom}(\tfrac{2}{2 + \nu})$.

Note that Theorem \ref{thm:fdd>1} holds for any $\varepsilon > 0$, however we cannot 
take $\varepsilon = 0$. To see the behavior near 0, one has to apply a time reversal, and the continuous mapping theorem.
First, we need some properties of the limit process.

For $a \in [0,1]$, $\nu \geq 0$, introduce the notation
\begin{equation*} \label{eq:def-h}
h_a(s) = 1 - a \left( \frac{1}{1-s} + \frac{\nu}{2} ( 1-a) \right)^{-1},
\end{equation*}
that is, $h_a$ is the generating function of a linear fractional 
distribution with mean $a$. We stress that $h_a$ does depend on both $a$ and $\nu$, 
however the latter parameter is fix, therefore we suppress in the notation.
Note that we also allow the degenerate case $\nu=0$, when $h_a(s)$ is the g.f.~of a Bernoulli distributed variable with parameter $a$.
Recall the important property of 
linear fractional distributions that 
\begin{equation*} \label{eq:linfrac-conv}
h_a ( h_b(s)) = h_{ab}(s), \quad \forall a,b \in [0,1].
\end{equation*}
With this notation $F$ defined in \eqref{eq:db-genfnc} can be expressed as
$F(s,t) = h_{e^{-t}}(s)$.

Let  us define the time-changed process
\begin{equation} \label{eq:defU}
U(t) = Z(\log t).
\end{equation}
Then 
for $0 < u < t$, and $x_0 \in \{1, 2, \ldots \}$
\begin{equation} \label{eq:U-cgf}
\begin{split}
\E \left[ s^{{ U}(t)} | { U}(u) = x_0 \right] & = 
\E \left[ s^{Z(\log t)} | Z(\log u) = x_0\right] \\
& = \left( \E \left[ s^{Z(\log t)} | Z(\log u) = 1 \right] \right)^{x_0} \\
& = \left( h_{e^{-(\log t - \log u)}}(s) \right)^{x_0} =
\left( h_{u/t} (s) \right)^{x_0}.
\end{split}
\end{equation}

We analyze the limiting process defined in \eqref{eq:defU}.
Put $p = \tfrac{2}{2+\nu}$, $q = 1-p$.
Define the generating function
\begin{equation*} \label{eq:def-g}
g_t(s) 
= \frac{ps}{1-sq} \, 
\frac{t^{-1} (1 - h_{t}(0) )}
{1 - h_{t}(0) q s}
= \sum_{x=1}^\infty 
p q^{x-1} t^{-1} ( 1 - h_{t}(0)^x) s^x.
\end{equation*}
Note that $t^{-1} (1 - h_{t}(0)) = \tfrac{p}{1 - q t}
\to p$, as $t \downarrow 0$.
Furthermore, $g_t$ is the generating function of
$V + W_t$, where $V \sim \text{Geom}(p)$,
$W_t + 1 \sim \text{Geom}( \tfrac{p}{1 - q t})$,
and $V$ and $W_t$ are independent.
Let $\p_\varepsilon$ denote the law of the process 
$(U(t)_{t \in [\varepsilon, 1]})$
under the 
initial distribution 
$U(\varepsilon) \sim \mathrm{Geom}(p)$.

The distribution $\mathrm{Geom}(p)$ is the 
extremal quasi-stationary distribution of the birth-and-death 
process $Z$; see Collet et al.~\cite[p.106]{ColletMartinez}.
This means that 
\[
\E_\varepsilon [ s^{U(1)} | U(1) > 0] = \frac{ps}{1-qs}
= \E_\varepsilon (s^{U(\varepsilon)}),
\]
i.e.~conditioned on nonextinction the distribution of $U(1)$ is the same 
as the initial distribution.
The second statement below is an extension of the quasi-stationary
property, which corresponds to $t = 1$.

\begin{lemma} \label{lemma:condU}
For $0 < u < t \leq 1$ and $x_0 \in \N$
the generating function of the conditional transition 
probabilities is given by
\begin{equation} \label{eq:k-def}
\E \left[ s^{U(t)} | U(u) = x_0, U(1) > 0 \right]
= \frac{(h_{u/t}(s))^{x_0} - 
(h_{u/t}(h_{t} (0) s))^{x_0} }
{1 - (h_{u}(0))^{x_0}} =: k_{u,t;x_0}(s).
\end{equation}
The family of laws with generating functions
$(g_t)_{t \in (0,1]}$ is an entrance law for 
the transition generating functions $k_{u,t;x}$, that is, 
for $\varepsilon \in (0,1]$ and for any $t \in [\varepsilon, 1]$,
\[
\E_\varepsilon \left[ s^{U(t)} | U(1) > 0 \right] = g_t(s).
\]
\end{lemma}

Entrance laws for homogeneous Markov chains were treated 
by Chung \cite[Definition I.1.1]{Chung} and 
by Rogers and Williams \cite[III.5.39]{RogersWilliams}. 
In \cite{RogersWilliams} the existence of a c\`adl\`ag 
Markov process on $(0,\infty)$ with the given transition
probabilities and entrance laws were proved.
In the general non-homogeneous setup the problem was 
investigated by Kolmogorov \cite{Kolmogorov}
(for the English translation see \cite{Kolmogorov2})
already in 1936, where entrance laws are called 
absolute probabilities. 
The existence of a Markov process on $(0,1]$ follows
simply from Kolmogorov's consistency theorem, while 
the path property follows after time reversal.

\begin{corollary} \label{cor:limit}
There exists a c\`adl\`ag
Markov process $(\widetilde U(t))_{t \in (0,1]}$
such that its transition probabilities are given by the 
generating functions 
\[ 
k_{u,t;x}(s) = \E [ s^{\widetilde U (t)} | \widetilde U(u) = x],
\]
and $\E (s^{\widetilde U(t)}) = g_t(s)$ for each $t \in (0,1]$.
\end{corollary}

As we noted earlier
$\lim_{t \downarrow 0} g_t(s) = \tfrac{ps}{1-qs} \tfrac{p}{1-qs}$, that is 
$\widetilde U(t) \stackrel{\mathcal{D}}{\longrightarrow} V + W$ 
as $t \downarrow 0$ where $V, W$ are independent, 
$V, W + 1 \sim \textrm{Geom}(p)$. However, contrary to the homogeneous 
case,  see \cite[III.39]{RogersWilliams}, $\widetilde U(t)$ as $t \downarrow 0$
does not converge in probability. Indeed, short calculation proves that 
for $\alpha \in (0,1)$ as $t \downarrow 0$
\[
k_{\alpha t, t; x_0}(s) =
\frac{(h_\alpha(s))^{x_0} - (h_\alpha(h_t(0)s))^{x_0}}
{1 - (h_{\alpha t}(0))^{x_0}} \longrightarrow \alpha^{-1} s (h_{\alpha}'(s))^{x_0 -1},
\quad s \in [0,1],
\]
that is the transition probabilities around 0 converge to a 
nontrivial distribution.
\smallskip 

As a consequence of the existence result above and Theorem \ref{thm:fdd>1}, using time reversal and the continuous mapping theorem we obtain the limiting behavior on $(0,1]$.

\begin{corollary} \label{cor:reverse}
Under the assumptions of Theorem \ref{thm:fdd>1},
\[
\mathcal{L} ( (X_{A(\frac{n}{t})} )_{t \geq 1} | X_{A(n)} > 0 ) 
\stackrel{\mathcal{D}}{\longrightarrow} 
\mathcal{L} (({Z}(-\log t))_{t \geq 1} | Z(0) > 0), \quad
\text{as $n\to\infty$}.
\]
\end{corollary}

\subsection{Nearly degenerate BPVEIs}\label{subsec:BPVE-immig}

Consider a BPVEI $(Y_n)_{n\in\N}$ defined in \eqref{eq:Y} and introduce the notation
\[
m_{n,k} \coloneqq \E[\varepsilon_n (\varepsilon_n - 1) \cdots (\varepsilon_n - k +1)], \quad k\in \N,
\]
for the factorial moments of the immigration in generation $n$.
In what follows, we suppose that $(Y_n)_{n\in\N}$ is a nearly degenerate BPVEI, that is, its offspring distribution satisfies (C1)--(C3). We further assume that one of the following conditions holds:
\begin{itemize}
\item[(C4)] $\lim_{n\to\infty, \fl_n < 1}\frac{m_{n,k}}{k!(1-\fl_n)} = \lambda_k$, 
$k = 1,2,\ldots, K$, for some $K \geq 2$, such that $\lambda_K=0$ and 
$\left(\frac{m_{n,k}}{k!|1-\fl_n|}\right)_{n}$ is bounded for each $k\leq K$,
\item[(C4')]  $\lim_{n\to\infty, \fl_n < 1} \frac{m_{n,k}}{k!(1-\fl_n)} = 
\lambda_k$, $k = 1,2,\ldots$ such that $\limsup_{n\to\infty}\lambda_n^{1/n} \leq 1$ and
$\left(\frac{m_{n,k}}{k!|1-\fl_n|}\right)_{n}$ is bounded for all $k$.
\end{itemize}

For BPVEIs the limiting processes are \emph{continuous time branching processes with immigration}, defined as follows. 
Each individual has an exponentially distributed lifetime with 
parameter $\alpha$, and upon death, it leaves a random number of offspring $\xi$. 
Furthermore, random number of immigrants $\varepsilon$ arrive at random times determined by a Poisson process of intensity $\beta$. 
All the quantities involved are independent.
The offspring and immigrant generating functions are
\begin{equation*}\label{eq:f(s)}
f(s) = \sum_{k=0}^\infty \p(\xi = k) s^k, \quad 
h(s) = \sum_{k=1}^\infty \p (\varepsilon = k) s^k, \quad |s|\leq 1,
\end{equation*}
where we may and do assume that $\p (\xi = 1) = 0$ and 
$\p(\varepsilon = 0) = 0$.

When there is no immigration we obtain the well-known class of continuous 
time branching processes; see \cite[Chapter III]{AthreyaNey}.
In particular, a birth-and-death process is a special continuous time branching 
process with
\[
\alpha = 1 + \nu, \quad f(s) = \frac{1}{1 + \nu}
\left( 1 + \frac{\nu}{2} + \frac{\nu}{2} s^2 \right).
\]
Continuous time branching processes with immigration are less studied, though they 
were introduced by Sevastyanov \cite{Sevastyanov} in 1957. 
Simple birth-and-death processes with immigration are discussed in detail by 
Allen \cite[Section 6.4.4]{Allen}, where immigrants arrive one by one.
For some recent references see Pakes, Chen, and Li \cite{LiChenPakes}, and Li and Li \cite{LiLi}.

Let $(W(t))_{t\geq - w}$ denote the resulting continuous time branching process with immigration on the time interval $t \in [-w, \infty)$, $w \geq 0$,
and introduce the notation $G(s,t) = \E(s^{W(t+u)} | W(u) =0)$, $t \geq 0$,
$s \in [0,1]$. Then $G(s,t)$ satisfies the Kolmogorov forward equation (see 
(2.5) in \cite{LiChenPakes}),
\begin{equation*}\label{eq:Kolmogorov-MBI}
\frac{\partial}{\partial t} G(s,t) = 
a(s) \frac{\partial}{\partial s} G(s,t) + b(s) G(s,t)
\end{equation*}
with boundary condition
\[
G(s,0) = 1,
\]
where
\begin{equation} \label{eq:def-ab}
\begin{split}
a(s) & =  \alpha \left(f(s) - s\right), \\
b(s) & = \beta \left(h(s) - 1\right).
\end{split}
\end{equation}

\begin{theorem}\label{thm:conv-immig}
Assume that the BPVEI $(Y_n)_{n\in\N}$ defined in \eqref{eq:Y} satisfies the conditions (C1)--(C3) and (C4) or (C4') holds. Then, for any $0<\varepsilon \leq 1$,
\[
\mathcal{L} ((Y_{A(nt)})_{t \geq  \varepsilon}) 
\stackrel{\mathcal{D}}{\longrightarrow} \mathcal{L} ((W(\log t))_{t \geq \varepsilon}),
\]
as $n\to \infty$, 
where $(W(u))_{u\geq \log \varepsilon}$ is a continuous time branching 
process with immigration with initial distribution having g.f.,
\[
\log f_Y(s) 
= 
\begin{cases}
-\sum_{k=1}^{\kappa} 
\frac{2^k\lambda_k}{\nu^k} \left( \log\left( 
1+\frac{\nu}{2}(1-s)\right) +
\sum_{i=1}^{k-1}\frac{\nu^i}{i2^i}
(s-1)^i\right), & \nu > 0, \\
\sum_{k=1}^{\kappa} \frac{\lambda_k}{k} (s-1)^k, & \nu = 0,
\end{cases}
\]
with parameters $\alpha = 1+\nu$, $\beta = \sum_{k=1}^{\kappa}(-1)^{k+1}\lambda_k$,
and
\[
\begin{split}
f(s) & = (1+\nu)^{-1} \left(1+\frac{\nu}{2} + \frac{\nu}{2}s^2\right), \\
h(s) & = 1 + \beta^{-1} \sum_{k=1}^{\kappa} \lambda_k (s-1)^k,
\end{split}
\]
offspring and immigration g.f.'s, where $\kappa=K-1$ or $\kappa=\infty$ depending on whether condition (C4) or (C4') holds. The transition generating function of $W$ is given by
\begin{equation*} \label{eq:G-def}
G_y(s, t) = \E \left[ s^{W(t + u)} | W(u) = y \right] = 
\frac{f_Y(s)}{f_Y(h_{e^{-t}}(s))} \left( h_{e^{-t}}(s) \right)^y.
\end{equation*}
\end{theorem}

Contrary to the conditional setup, the limiting process is stationary. Indeed,
\[
\begin{split}
\sum_{k=0}^\infty \E [ s^{W(t)} | W(\log \varepsilon) = k] f_Y[k] 
& = 
\sum_{k=0}^\infty \frac{f_Y(s)}{f_Y(h_{e^{-(t-\log \varepsilon)}}(s))} 
\left( h_{e^{-(t-\log \varepsilon)}}(s) \right)^{k} 
f_Y[k] \\
& = \frac{f_Y(s)}{f_Y(h_{e^{-(t-\log \varepsilon)}}(s))} 
f_Y(h_{e^{-(t-\log \varepsilon)}}(s)) = f_Y(s).
\end{split}
\]


If $K=2$ the limiting process is a simple birth-and-death process with immigration with birth, death, and immigration rates
$\frac{\nu}{2}$, $\left(1+\frac{\nu}{2}\right)$, $\lambda_1$, respectively;
see \cite[Section 6.4.4]{Allen}.


Similarly to Theorem \ref{thm:fdd>1}, if $\nu = 0$, reproduction occurs according to a simple death process with parameter $\mu = 1$.

As in the conditional setup, in the limit theorem we cannot choose 
$\varepsilon = 0$. To understand the behavior near 0, we need time reversal.

\begin{corollary} \label{corr:imm}
Under the assumptions of Theorem \ref{thm:conv-immig},
\[
\mathcal{L} ((Y_{A(\frac{n}{t})})_{t \geq 1}) 
\stackrel{\mathcal{D}}{\longrightarrow} \mathcal{L}((W(-\log t))_{t \geq 1}), \quad
\text{as $n\to \infty$}.
\]
\end{corollary}

\section{Proofs}\label{sec:proofs}

\subsection{Preparation}

Recall that $f_n(s) = \sum_{k=0}^\infty f_n[k] s^k = \E (s^{\xi_n})$ denotes the offspring
g.f.~in generation $n$. We introduce the notations,
\[
f_{n,n}(s) = s, \quad
f_{j,n}(s) \coloneqq f_{j+1} \circ \dots \circ f_n(s),
\]
for $j <n$, put $\fl_n = \E (\xi_n)$, for the offspring mean and let
\[
\fl_{n,n} = 1, \quad \fl_{j,n} \coloneqq \fl_{j+1} \cdots \fl_n,
\]
for $j<n$.

We analyze the g.f.~of the underlying processes. For a given g.f.~$f$
with mean $\fl$ and $f''(1) < \infty$ define the \emph{shape function},
see e.g.~\cite{Kersting}, as
\begin{equation} \label{eq:def-shape}
\varphi(s) = \frac{1}{1-f(s)} - \frac{1}{\fl(1-s)}, \quad 0\leq s\leq 1,
\quad \varphi(1) = \frac{f''(1)}{2\fl^2}.
\end{equation}
Let $\varphi_j$ be the shape function of $f_j$. Using the definition of the shape function, it is easy to check that
\begin{equation}\label{eq:shape}
\frac{1}{1 - f_{j,n} (s)} = 
\frac{1}{\fl_{j,n}(1 - s)} + \varphi_{j,n}(s), 
\quad \varphi_{j,n}(s)\coloneqq\sum_{k=j+1}^{n} 
\frac{\varphi_k(f_{k,n}(s))}{\fl_{j,k-1}}.
\end{equation}

The following lemmas are frequently used in our proofs. The first one is a consequence of the Silverman--Toeplitz theorem, while the second one follows from a straightforward calculation. For their proofs we refer to \cite{KevKub}.

\begin{lemma}[{\cite[Lemma 2]{KevKub}}] \label{lemma:anjk}
Let $(\fl_n)_{n\in\N}$ be a sequence of positive real numbers satisfying (C1), 
and define 
\[
a_{n,j}^{(k)}=(1-\fl_j)\prod_{i=j+1}^{n}\fl_i^k
= (1 - \fl_j) \fl_{j,n}^k, \quad n,j,k\in\N, \quad j\leq n-1,
\]
and $a_{n,n}^{(k)} = 1 - \fl_n$.
If $(x_n)_{n\in\N}$ is bounded and 
$\lim_{n\to\infty, \fl_n < 1} x_n = x \in \R$, then for all $k\in\N$,
\begin{equation*} \label{eq:x-conv}
\begin{split}
\lim_{n \to \infty} \sum_{j=1}^n a_{n,j}^{(k)} x_j = \frac{x}{k}, \quad
\lim_{n \to \infty} \sum_{j=1}^n | a_{n,j}^{(k)}| x_j = \frac{x}{k}.
\end{split}
\end{equation*}
\end{lemma}

\begin{lemma}[{\cite[Lemma 3]{KevKub}}]\label{lemma:phi}
Let $\varphi_n$ be the shape function of $f_n$. Then
under the conditions (C1)--(C3) with $\nu > 0$,
\[
\lim_{n\to\infty, \fl_n < 1}
\sup_{s \in [0,1]}
\frac{|\varphi_n(1)-\varphi_n(s)|}{1-\fl_n} = 0,
\]
and the sequence 
$\sup_{s \in [0,1]} \frac{|\varphi_n(1)-\varphi_n(s)|}{|1-\fl_n|}$
is bounded.
\end{lemma}

\subsection{Proofs of Subsection \ref{subsec:BPVE}}

\begin{lemma}\label{lemma:funclimit1}
Assume that condition (C1)--(C3) are satisfied for a BPVE $(X_n)_{n\in \N}$ defined in \eqref{eq:X}.
Then, for any $0 < u \leq t$,
\[
\lim_{n \to \infty}
\E\left[s^{X_{A(nt)}}  | X_{A(nu)} = 1 \right] =
1 - \frac{u}{t} \left(
\frac{1}{1-s} + \frac{\nu}{2} \left(1 - \frac{u}{t}\right)
\right)^{-1} = h_{u/t}(s),
\]
that is, $\mathcal{L}(X_{A(nt)} | X_{A(nu)} = 1 )$ converges 
in distribution to a linear fractional distribution with 
generating function $h_{u/t}(s)$.
\end{lemma}

\begin{proof}
By \eqref{eq:shape}, we have
\[
\begin{split}
& (1 - f_{j,n}(s))^{-1} 
= \frac{1}{\fl_{j,n} }
\left( \frac{1}{1-s} + \sum_{k=j+1}^n \fl_{k-1,n} \varphi_k (f_{k,n} (s)) 
\right) \\
& = \frac{1}{\fl_{j,n} }
\left( \frac{1}{1-s} + \sum_{k=j+1}^n \fl_{k-1,n} \varphi_k (1)
- \sum_{k=j+1}^n \fl_{k-1,n} (\varphi_k (1) - \varphi_k (f_{k,n} (s))) 
\right).
\end{split}
\]
By Lemmas \ref{lemma:anjk} and \ref{lemma:phi}, for the third term 
we obtain
\[
\begin{split}
\left|\sum_{k=j+1}^{n} \fl_k \frac{\varphi_k(1) - 
\varphi_k(f_{k,n}(s))}{1-\fl_k} a_{n,k}^{(1)} \right| 
\leq \sum_{k=j+1}^{n} \left| 
\frac{\varphi_k(1) - \varphi_k(f_{k,n}(s))}{1-\fl_k} \right| 
| a_{n,k}^{(1)} | \to 0.
\end{split}
\]
Recalling from \eqref{eq:def-shape} that 
$\varphi_k(1) = f_k''(1) / (2 \overline f_k^2)$, we get
\[
\begin{split}
\left| \sum_{k=j+1}^{n}\frac{1}{\fl_k} \frac{f_k''(1)}{1-\fl_k}a_{n,k}^{(1)}-
\nu \sum_{k=j+1}^{n}a_{n,k}^{(1)}\right| & 
= \left|\sum_{k=j+1}^{n} a_{n,k}^{(1)}\left(\frac{1}{\fl_k} 
\frac{f_k''(1)}{1-\fl_k}-\nu\right) \right|\\ 
& \leq \sum_{k=1}^{n} |a_{n,k}^{(1)}|
\left| \frac{1}{\fl_k} \frac{f_k''(1)}{1-\fl_k}-\nu\right|\to 0,
\end{split}
\]
where
\[
\sum_{k=j+1}^{n}a_{n,k}^{(1)} = 
 \sum_{k=j+1}^{n} \left (\fl_{k,n} - \fl_{k-1,n} \right) 
= \left(1-\fl_{j,n}\right).
\]
By definition \eqref{eq:A(n)} of $A(nt)$,
\begin{equation}\label{eq:E-conv}
\overline f_{A(nu), A(nt)} \to \frac{u}{t},
\end{equation}
as $n\to\infty$, thus,
substituting $j = A(nu)$ and $n = A(nt)$, by \eqref{eq:E-conv},
\[
\left( 1 - f_{A(nu), A(nt)}(s) \right)^{-1} \to
\frac{t}{u} \left(
\frac{1}{1-s} + \frac{\nu}{2} \left(1 - \frac{u}{t}\right)
\right).
\]
After rearranging, the result follows.
\end{proof}

Now we prove the properties of the limit process.

\begin{proof}[Proof of Lemma \ref{lemma:condU}]
Recall that $p = \tfrac{2}{2+\nu} = 1 - q$. 
First we calculate the conditional transition generating functions. 
Using the Markov property,
\[
\begin{split}
& \E \left[ s^{U(t)} | U(u) = x_0, U(1) > 0 \right] 
= \sum_{x=1}^\infty s^x \p ( U(t) = x | U(u) = x_0, U(1) > 0) \\
& = \sum_{x=1}^\infty s^{x} 
\frac{\p ( U(t) = x, U(u) = x_0, U(1) > 0)}{\p ( U(u) = x_0, U(1) > 0)} \\
& = \sum_{x=1}^\infty s^{x} 
\frac{\p ( U(t) = x | U(u) = x_0) \p ( U(1) > 0 | U(t) = x)}
{\p ( U(1) > 0 | U(u) = x_0 )} \\
& = \sum_{x=1}^\infty s^{x} 
\frac{\p ( U(t) = x | U(u) = x_0) [ 1 - ( h_t(0))^x ]}
{1 - (h_{u}(0))^{x_0}} \\
& = \frac{(h_{u/t}(s))^{x_0} - 
(h_{u/t}(h_{t} (0) s))^{x_0} }
{1 - (h_{u}(0))^{x_0}}.
\end{split}
\]

In the calculations below we use the identity
\begin{equation} \label{eq:h-iden}
\frac{p h_a(s)}{1 - qh_a(s) } = 
1 - \frac{ a ( 1 + \frac{\nu}{2})}{(1-s)^{-1} + \frac{\nu}{2}},
\quad a \in [0,1], \ s \in [0,1].
\end{equation}
The proof of \eqref{eq:h-iden} is a long but straightforward calculation.
By the law of total probability, \eqref{eq:U-cgf}, and \eqref{eq:h-iden},
\begin{equation} \label{eq:Usurv}
\begin{split}
\p_\varepsilon ( U(1) > 0) 
& = \sum_{x= 1}^\infty \p ( U(\varepsilon) = x) \, \p (U(1) > 0 | U(\varepsilon ) = x) \\
& = \sum_{x=1}^\infty p q^{x-1}  (1 - (h_{\varepsilon}(0))^x) 
= 1 - \frac{p h_{\varepsilon}(0)}{1 - q h_{\varepsilon}(0)}
= \varepsilon.
\end{split}
\end{equation}
Note that this means that the distribution of the extinction time 
of the birth-and-death process $Z$ is exponential; see e.g.~formula (3.3)
in \cite{ColletMartinez}.

Using \eqref{eq:Usurv}, \eqref{eq:k-def}, and \eqref{eq:h-iden} again, 
we obtain
\[
\begin{split}
& \E_\varepsilon \left[ s^{U(t)} | U(1) > 0 \right] 
= \sum_{x=1}^\infty \p_\varepsilon(U(\varepsilon)= x | U(1) > 0)
\E \left[ s^{U(t)} | U(1) > 0, U(\varepsilon ) = x \right] \\
& = \sum_{x=1}^\infty 
\frac{\p_\varepsilon(U(\varepsilon)= x , U(1) > 0)}
{\p_\varepsilon( U(1) > 0)}
k_{\varepsilon, t; x}(s) \\
& = \varepsilon^{-1}
\sum_{x=1}^\infty p q^{x-1}
\left[ (h_{\varepsilon/t}(s))^x
- (h_{\varepsilon/t}(h_{t}(0)s))^x \right] \\
& = \varepsilon^{-1} 
\left[ \frac{ p h_{\varepsilon /t}(s)}
{ 1 - q h_{\varepsilon /t}(s)} -
\frac{ p h_{\varepsilon /t}(h_{t}(0)s)}
{ 1 - q h_{\varepsilon /t}(h_{t}(0)s)} \right] \\
& = 
\frac{ps}{1- qs} \, \frac{t^{-1} ( 1 - h_{t}(0))}
{1 - q s h_{t}(0)} = g_t(s),
\end{split}
\]
proving the statement.
\end{proof}

Next, we prove the existence of a c\`adl\`ag version of the 
limit process on $(0,1]$.

\begin{proof}[Proof of Corollary \ref{cor:limit}]
The existence of a Markov process $\widetilde U$ 
on the product space 
$\N^{(0,1]}$ with the prescribed distribution follows 
from Kolmogorov consistency theorem. 
Furthermore, for any $\varepsilon > 0$ the process has 
a c\`adl\`ag version on $[\varepsilon, 1]$. 
This follows from the fact that $\widetilde U$ on $[\varepsilon, 1]$
is a time-transformed birth-and-death process.
We construct a c\`adl\`ag version on $(0,1]$.

Using time reversal 
$\overline U(s) = \widetilde U(1-s-)$, $s \in [0,1)$,
where $U(s - )$ stands for the left limit in $s$,
defines a Markov chain with transition probabilities 
\begin{equation} \label{eq:tran-overline}
\p ( \overline U(t) = y | \overline U(u) = x) = 
\frac{g_{1-t}[y]}{g_{1-u}[x]} \p ( \widetilde U(1-u-) = x | 
\widetilde U(1-t-) = y),
\end{equation}
and marginals 
\[
\E (s^{\overline U(t)} )= g_{1-t}(s).
\]
Thus, we have a Markov process $(\overline U(t): t \in [0,1))$
with initial distribution $\textrm{Geom}(\tfrac{2}{2+\nu})$
and transition probabilities given in \eqref{eq:tran-overline}. The process
has a c\`adl\`ag version on the interval $[0,1-n^{-1}]$, and can be continued 
for $t \in [1-n^{-1}, 1-(n+1)^{-1}]$. Concatenating these processes we obtain a c\`adl\`ag version on $[0,1)$. 
From $\overline U$ a time reversal provides the required 
version of $\widetilde U$.
\end{proof}

The first step towards Theorem \ref{thm:fdd>1}
is to show the convergence of finite dimensional distributions.

\begin{proposition} \label{prop:fdd}
Assume that the BPVE $(X_n)_{n\in \N}$ satisfies conditions (C1)--(C3) and for $k \geq 0, \ell \geq 0$ let 
$0 < t_{-k} < t_{-k+1} < \ldots < t_{-1} < t_0 = 1 < t_{1} < 
\ldots < t_{\ell} < \infty$, and $x_i \in \N$, $i = -k, \ldots, \ell$. 
Then
\[
\begin{split}
& \lim_{n \to \infty}
\p ( X_{A(n t_i)} = x_i, i= -k, \ldots,  \ell | 
X_{A(n)} > 0) \\
& = \p_{t_{-k}} ( U(t_i) = x_i, i= -k, \ldots, \ell | U(1) > 0),
\end{split}
\]
where $\p_{t_{-k}}$ denotes the law of $(U(t))_{t \geq t_{-k}}$ with 
initial distribution $U(t_{-k}) \sim \textrm{Geom}(\tfrac{2}{2+\nu})$.
\end{proposition}

\begin{proof}
Recall the notation $p = \tfrac{2}{2 + \nu} = 1-q$.
First we show that 
\begin{equation} \label{eq:fdd-aux1}
\begin{split}
& \lim_{n \to \infty}
\p ( X_{A(n t_i)} = x_i, i= -k, \ldots,  \ell | 
X_{A(n)} = x_0) \\
& =
q^{x_{-k} - x_0} (t_{-k})^{-1} \,
\p ( { U}(t_i) = x_i, i = -k, \ldots, \ell | { U}(t_{-k}) = x_{-k} ). 
\end{split}
\end{equation}

By the Markov property
\[
\begin{split}
& \p ( X_{A(n t_i)} = x_i, i= -k,  \ldots, \ell | 
X_{A(n)} = x_0) \\
& = \frac{\p(X_{A(n t_{-k})} = x_{-k})}{\p(X_{A(n)} = x_0)}
\prod_{i=-k+1}^\ell 
\p(X_{A(nt_i)}=x_i | X_{A(nt_{i-1})}=x_{i-1}).
\end{split}
\]
By Lemma \ref{lemma:funclimit1} and \eqref{eq:U-cgf}
\[
\lim_{n \to \infty} \p(X_{A(nt_i)}=x_i | X_{A(nt_{i-1})}=x_{i-1})
= \p ( { U}(t_i) = x_i | { U}(t_{i-1}) = x_{i-1} ).
\]
Furthermore, by Theorem 1 of \cite{KevKub} for all $x \in \N$,  
\begin{equation} \label{eq:thm1-lim}
\lim_{n\to\infty} \p(X_{n} = x| X_n>0) 
= p q^{x-1},
\end{equation}
and by Lemma 4 in \cite{KevKub},
\[
\frac{\p(X_{A(ns)} > 0)}{\p(X_{A(n)} > 0 )} 
= \frac{1 - f_{0, A(ns)}(0)}{1 - f_{0,A(n)}(0)} \sim 
\frac{\overline f_{0,A(ns)} }{\overline f_{0,A(n)} } \sim s^{-1},
\]
where the last asymptotic equality follows from \eqref{eq:A(n)}.
Hence,
\[
\begin{split}
& \lim_{n \to \infty} 
\frac{\p(X_{A(n t_{-k})} =x_{-k})}{\p(X_{A(n)} = x_0)} \\
& = \lim_{n\to \infty} 
\frac{\p(X_{A(n t_{-k})} = x_{-k})}{\p(X_{A(n t_{-k})} > 0)} 
\cdot 
\frac{\p(X_{A(n t_{-k})} > 0)}{\p(X_{A(n)} > 0)} 
\cdot \frac{\p(X_{A(n)} > 0)}{\p(X_{A(n)} = x_0)} \\
& = q^{x_{-k}-x_0} (t_{-k})^{-1},
\end{split}
\]
thus \eqref{eq:fdd-aux1} follows.

Using \eqref{eq:fdd-aux1}, \eqref{eq:thm1-lim}, and \eqref{eq:Usurv} we have
\begin{equation*}\label{eq:findim}
\begin{split}
& \lim_{n \to \infty}
\p(X_{A(nt_i)} = x_i, i =-k,\ldots, \ell | X_{A(n)} > 0) \\
& = \lim_{n \to \infty} 
\frac{\p ( X_{A(n)} = x_0)}{\p ( X_{A(n)} > 0) }
\p( X_{A(nt_i)} = x_i, i = -k,\ldots, \ell | X_{A(n)} = x_0) \\
& = p q^{x_{-k} -1} (t_{-k})^{-1}
\p ( { U}(t_i) = x_i, i = -k, \ldots, \ell | { U}(t_{-k}) = x_{-k}) \\
& = \p_{t_{-k}} ( { U}(t_i) = x_i, i = -k, \ldots,  \ell | U(1)  > 0)
\frac{p q^{x_{-k} -1} \p ( U(1) > 0)}
{ t_{-k} \p ( U(t_{-k} ) = x_{-k} )} \\
& = \p_{t_{-k}} ( { U}(t_i) = x_i, i = -k, \ldots,  \ell | U(1)  > 0),
\end{split}
\end{equation*}
proving the statement.
\end{proof}

\begin{proof}[Proof of Theorem \ref{thm:fdd>1}]
The convergence of finite dimensional distributions follows from 
Proposition \ref{prop:fdd}. 
This together with tightness, by Theorem 3.7.8 in Ethier and Kurtz \cite{EthKur} implies
convergence. Now, let us prove tightness.

We prove the classical tightness criteria in Theorem 16.8 and in the Corollary after it in Billingsley \cite{Billingsley}.
First we recall the standard notation from \cite{Billingsley}. For 
$x \in D[0,\infty)$, $T \subset [0,\infty)$ let
\[
w(x, T) = \sup_{s,t \in T} |x(t) - x(s)|,
\]
and for $m > 0$, $\delta > 0$
\begin{equation*} \label{eq:def-w'}
w_m'(x,\delta) = \inf \max_{1\leq i\leq v} w(x,[t_{i-1},t_i)),
\end{equation*}
where the infimum extends over all partitions $[t_{i-1}, t_i)$, $i = 1,\ldots, v$,
of $[ 0, m )$ such that $t_i - t_{i-1} > \delta$ for all $i$.

We have to show that (condition (i) in \cite[Corollary of Theorem 16.8]{Billingsley})
\begin{equation} \label{eq:bill1}
\lim_{a\to\infty} \limsup_{n\to\infty} 
\p ( |X_{A(nt)}|\geq a | X_{A(n)} > 0) = 0, \quad \forall t \geq \varepsilon,
\end{equation}
and that (condition (ii) in \cite[Theorem 16.8]{Billingsley}) for all $m > 0$
\begin{equation} \label{eq:bill2}
\lim_{\delta \to 0} \limsup_{n\to\infty}
\p(w_m'((X_{A(nt)})_{t\in [\varepsilon, m]}, \delta) \geq \eta| X_{A(n)} > 0) = 0.
\end{equation}
From the convergence of finite dimensional distributions \eqref{eq:bill1} follows 
immediately.
\smallskip 

We prove \eqref{eq:bill2}. For $\varepsilon < 1$ we have 
$\{ X_{A(n \varepsilon)} > 0 \} \supset \{ X_{A(n)} > 0 \}$ and 
by \eqref{eq:E-conv} and Lemma 4 in  \cite{KevKub} 
$\tfrac{\p( X_{A(n\varepsilon)} > 0)}{\p( X_{A(n)} > 0)} \to \varepsilon^{-1}$.
Thus for any event $F$
\[
\begin{split}
\limsup_{n \to \infty} \p ( F | X_{A(n)} > 0) & \leq 
\limsup_{n \to \infty} \frac{ \p ( F \cap \{ X_{A(n\varepsilon)} > 0 \} ) }
{\p (X_{A(n\varepsilon)} > 0)} \frac{\p( X_{A(n\varepsilon)} > 0)}{\p( X_{A(n)} > 0)} \\
& \leq \varepsilon^{-1} \limsup_{n \to \infty} \p ( F | X_{A(n \varepsilon)} > 0).
\end{split}
\]
Therefore, it is enough to prove \eqref{eq:bill2}
with the condition $X_{A(n\varepsilon)} > 0$.

Next we show that we can assume bounded sample paths.
Lemma 4 in \cite{KevKub} and \eqref{eq:E-conv} imply 
that 
\begin{equation} \label{eq:E-lim}
\begin{split}
\lim_{n\to\infty} \E [ X_{A(nm)} | X_{A(n\varepsilon)} > 0 ] 
& = \lim_{n \to \infty} \frac{\overline f_{0,A(nm)}}{1 - f_{0,A(n\varepsilon)}} \\
& = \lim_{n \to \infty} 
\frac{\overline f_{0,A(n\varepsilon)}}{1 - f_{0,A(n\varepsilon)}}
\frac{\overline f_{0,A(nm)}}{ \overline f_{0,A(n\varepsilon)}}
= \frac{2+\nu}{2} \frac{\varepsilon}{m} .
\end{split}
\end{equation}
Fix $K > 0$ large, and define the stopping time (with respect to the natural 
filtration)
$\tau = \min \{ k: \, A(n\varepsilon) \leq k \leq A(nm), \, X_k > K \}$, with the 
convention $\min \emptyset = \infty$.
By the law of total expectation and the branching property, we have
\[
\begin{split}
& \E[ X_{A(nm)} \ind_{\{A(n\varepsilon)\leq \tau \leq A(nm)\}}|X_{A(n\varepsilon)} > 0 ] \\
& = \sum_{k=A(n\varepsilon)}^{A(nm)}\E [X_{A(nm)} | \tau = k] 
\p(\tau = k | X_{A(n\varepsilon)} > 0) \\
& \geq \sum_{k=A(n\varepsilon)}^{A(nm)} K \fl_{k, A(nm)} 
\p(\tau = k | X_{A(n\varepsilon)} > 0) \\
& \geq \frac{K}{C_0} 
\fl_{A(n\varepsilon), A(nm)} 
\p\left( \max_{A(n\varepsilon)\leq k \leq A(nm)} X_k > K \Big| 
X_{A(n\varepsilon)} > 0 \right),
\end{split}
\]
where 
\begin{equation}\label{eq:C_0}
C_0 := \prod_{n=1, \overline f_n > 1}^\infty \overline f_n < \infty.
\end{equation}
At the last inequality we used that typically $\overline f_n < 1$, and for those $n$
when the converse inequality holds, the contribution is bounded by $C_0$.
Thus by \eqref{eq:E-lim} and \eqref{eq:E-conv}
for any $\varepsilon > 0$, and $m > 0$
\begin{equation}\label{eq:X<K}
\lim_{K\to\infty} \limsup_{n\to\infty} 
\p \left( \max_{A(n\varepsilon) \leq k \leq A(nm)} X_k > K \Big| 
X_{A(n\varepsilon)} > 0 \right) = 0.
\end{equation}
This means that we can indeed restrict to bounded sample paths.
\smallskip

Since each jump has at least size 1, if 
$w'_m((X_{A(nt)})_{t \in [\varepsilon, m]}, \delta) > 0$ then 
the process has either at least 
$N = \lfloor \tfrac{m-\varepsilon}{\delta} \rfloor + 1$
jumps, where $\lfloor \cdot \rfloor$ stands for the lower integer part, 
or at least three jumps on an interval $[A(nt_0), A(n(t_0 + \delta))]$ 
for some $t_0 \in (\varepsilon, m)$.
Fix $K > 0$ large, and introduce the events
\begin{equation*} \label{eq:def-CD}
\begin{split}
& B_n = \{ \max_{A(n\varepsilon) \leq k \leq A(nm)} X_k \leq K \}, \\
& C_n = \{ (X_{A(nt)})_{t \in [\varepsilon, m]} \text{ has at least $N$ jumps} \}, \\
& D_n = \{ \exists t_0 \in [\varepsilon, m] \text{ such that } 
(X_{t}) \text{ has at least 3 jumps in $[A(nt_0), A(n(t_0 + \delta))]$} \}.
\end{split}
\end{equation*}
Then for any $\eta \in (0,1)$
\begin{equation} \label{eq:w'bound}
\begin{split}
& \p(w'((X_{A(nt)})_{t\in [\varepsilon,m]}, \delta) \geq \eta | X_{A(n \varepsilon)} > 0) \\
& \leq \p( C_n \cap B_n | X_{A(n\varepsilon)} > 0 ) +
\p( D_n \cap B_n | X_{A(n \varepsilon)} > 0 ) + \p( B_n^c | X_{A(n \varepsilon)} > 0 ). 
\end{split}
\end{equation}
For the third term we have by \eqref{eq:X<K}
\begin{equation} \label{eq:wbound-3}
\lim_{K \to \infty} \limsup_{n \to \infty} \p ( B_n^c | X_{A(n\varepsilon)} > 0 ) = 0. 
\end{equation}

To handle the first two terms in \eqref{eq:w'bound} we have to understand the 
jump probabilities. For a discrete time process a jump means that $X_k \neq X_{k+1}$.
Since $\frac{f_n''(1)}{|1-\fl_n|}$ is bounded by (C2), with $f_n[1] = \p(\xi_n = 1)$, 
we have with some $s \in (0,1)$
\begin{equation} \label{eq:aux-f-n-bound}
\begin{split}
1 - f_n[1] & = 1  - f_n'(0) = f_n'(1) - f_n'(0) + 1 - f_n'(1) \\
& = f_n''(s) + 1 - f_n'(1) \leq c | 1 - f_n'(1)|,
\end{split}
\end{equation}
for some $c > 0$ and for $n$ large enough. Here, and later on, $c$ 
stands for a positive constant, whose value does not depend on relevant 
quantities, and can change from line to line.
Thus for $0 < x_k \leq K$,
\begin{equation}\label{eq:prob-jump}
\begin{split}
\p(X_{k+1} \neq X_k | X_k = x_k) 
& \leq 1 - (f_{k+1}[1])^{x_k} = 1 - (1 - (1 - f_{k+1}[1]))^{x_k} \\
& \leq 1 - \left( 1 - 2 x_k (1 - f_{k+1}[1]) \right) 
\leq c K |1-\fl_{k+1}|,
\end{split}
\end{equation}
by \eqref{eq:aux-f-n-bound}.

Next consider the event $C_n \cap B_n$.
Note that
\[
\{X_{k_i} \neq X_{k_i+1}, i=1,\dots, N\} \subset 
\{X_{k_i} > 0, i=1,\dots, N\} \subseteq \{X_{k_N} > 0\},
\]
hence
\[
\begin{split}
& \p ( B_n \cap C_n | X_{A(n\varepsilon)} > 0) \\
& = \p(\exists \,
A(n\varepsilon) \leq k_1 <  \ldots < k_N \leq A(nm): 
X_{k_i} \neq X_{k_i+1}, i = 1,\ldots, N, 
B_n | X_{A(n\varepsilon)} > 0) \\
& \leq \sum_{A(n\varepsilon) \leq k_1 < \ldots <k_N \leq A(nm)} 
\p(X_{k_i} \neq X_{k_i + 1}, i=1,\ldots, N, B_n | X_{A(n\varepsilon)} > 0).
\end{split}
\]
For $A(n \varepsilon) \leq k_1 < \ldots < k_N \leq A(nm)$
\[
\begin{split}
& \p(X_{k_i} \neq X_{k_i + 1}, i=1,\dots, N, B_n | X_{A(n\varepsilon)} > 0) \\
& \leq \p(X_{k_i} \neq X_{k_i + 1}, X_{k_i}, X_{k_i+1}\leq K, 
i=1,\dots, N | X_{A(n\varepsilon)} > 0)\\
& = \sum_{\substack{1\leq x_{k_1}, \dots, x_{k_N} \leq K, \\ x_{k_i} \neq x_{k_i+1}}}
\p(X_{k_i} = x_{k_i}, X_{k_i + 1} = x_{k_i+1}, i=1,\dots, N | X_{A(n\varepsilon)} > 0).
\end{split}
\]
Now let us compute the above expression for given $1\leq x_{k_i} \leq K$ satisfying the condition $x_{k_i} \neq x_{k_i + 1}$. 
By the Markov property and \eqref{eq:prob-jump} we have
\[
\begin{split}
& \p(X_{k_i} = x_{k_i}, X_{k_i + 1} = x_{k_i+1}, i=1,\dots, N | 
X_{A(n\varepsilon)} > 0) \\
& = \p(X_{k_1} = x_{k_1} | X_{A(n\varepsilon)} > 0) 
\p(X_{k_1 + 1} = x_{k_1+1} | X_{k_1} = x_{k_1}) \cdots \\
& \phantom{=} \times 
\p(X_{k_N + 1} = x_{k_N + 1} | X_{k_N} = x_{k_N}) \\
& \leq 1\cdot cK |1 -\fl_{k_1}| \cdot 1 \cdot c K |1-\fl_{k_2}| \cdots 
c K |1-\fl_{k_N}|.
\end{split}
\]
Hence,
\begin{equation} \label{eq:Njump}
\begin{split}
& \p(X_{k_i} \neq X_{k_i + 1}, i=1,\dots, N, 
\max_{n \varepsilon \leq k \leq n m} X_k \leq K |
X_{A(n\varepsilon)} > 0) \\
& \leq \sum_{\substack{1\leq x_{k_1}, \dots, x_{k_N} \leq K, \\ x_{k_i} \neq x_{k_i+1}}}
\p(X_{k_i} = x_{k_i}, X_{k_i + 1} = x_{k_i+1}, i=1,\dots, N | X_{A(n\varepsilon)} > 0) \\
& \leq (cK^2(K-1))^N |1-\fl_{k_1}| \cdots |1-\fl_{k_N}|,
\end{split}
\end{equation}
and thus
\[
\begin{split}
& \p ( B_n \cap C_n | X_{A(n\varepsilon)} > 0) \\
& \leq \sum_{A(n\varepsilon) \leq k_1 < \ldots <k_N \leq A(nm)} 
\p(X_{k_i} \neq X_{k_i + 1}, i=1,\ldots, N, B_n | X_{A(n\varepsilon)} > 0) \\
& \leq (cK^3)^N \sum_{A(n\varepsilon) \leq k_1 < \ldots <k_N \leq A(nm)}
|1-\fl_{k_1}| \cdots |1-\fl_{k_N}| \\
& \leq (cK^3)^N \frac{1}{N!} \left(\sum_{k=A(n\varepsilon)}^{A(nm)} |1-\fl_k|\right)^N
\leq \frac{C(K)^N}{N!},
\end{split}
\]
where $C(K)$ stands for a positive constant depending only on $K$, whose actual value 
can be different at each appearance. At the last inequality we used that 
\begin{equation}\label{eq:sum-1-fl}
\begin{split}
\sum_{k=A(n\varepsilon)}^{A(nm)} |1-\fl_k| & =
\sum_{k=A(n\varepsilon)}^{A(nm)} (1-\fl_k) + 2 \sum_{k=A(n\varepsilon)}^{A(nm)} (1-\fl_k)_- \\
& \leq - \log \fl_{A(n\varepsilon) - 1, A(nm)}
+ 2 \sum_{k=A(n\varepsilon)}^\infty (1-\fl_k)_-,
\end{split}
\end{equation}
where the first term tends to $\log m - \log \varepsilon$ by \eqref{eq:A(n)} and the second goes to 0 by (C1) as $n\to\infty$.
Since $N = N(\delta) \to \infty$ as $\delta \to 0$ we obtain
\begin{equation} \label{eq:BC-bound}
\lim_{\delta \to 0}
\limsup_{n \to \infty} \p (C_n \cap B_n | X_{A(n \varepsilon)} > 0) = 0.
\end{equation}
\smallskip

Finally, we bound the probability of three close jumps.
Consider a fixed $t_0$. By \eqref{eq:Njump} and \eqref{eq:sum-1-fl} we have
\[
\begin{split}
& \p ( (X_k)_{A(n t_0) \leq k \leq A(n(t_0+\delta))} \text{ has three jumps},
B_n |X_{A(n \varepsilon)} >0) \\
& \leq \sum_{A(nt_0) \leq k_1 <k_2 < k_3 \leq A(n (t_0+\delta))} 
\p (X_{k_i} \neq X_{k_i + 1}, i=1,2, 3, B_n | X_{A(n \varepsilon)} > 0) \\
& \leq (cK^3)^3 \sum_{A(nt_0) \leq k_1 <k_2 < k_3 \leq A(n (t_0+\delta))} 
|1-\fl_{k_1}| |1-\fl_{k_2}| |1-\fl_{k_3}| \\
& \leq (cK^3)^3 \frac{1}{3!} \left(
\sum_{k=A(n t_0)}^{A(n (t_0+\delta))} |1-\fl_k|
\right)^3 \\
&\leq
(c K^3)^3 \frac{1}{3!} \left(-\log \fl_{A(n t_0)-1,A(n(t_0+\delta))} + 2 \sum_{k=A(n t_0)}^{A(n(t_0+\delta))} (1-\fl_k)_- \right)^3,
\end{split}
\]
where the first term tends to 
\[
\log \left(1 + \frac{\delta}{t_0}\right) \leq  \frac{\delta}{t_0}
\]
and the second to 0 as $n \to\infty$ by (C1) and the fact that $A(n) \to\infty$.
Therefore,
\begin{equation} \label{eq:d-bound1}
\limsup_{n \to \infty} \p ( (X_k)_{A(n t_0) \leq k \leq A(n(t_0+\delta))} 
\text{ has three jumps}, B_n |X_{A(n \varepsilon)} >0) \leq C(K) \delta^3. 
\end{equation}
Since 
$[\varepsilon, m] \subset \cup_{\ell = 0}^{\lfloor (m-\varepsilon)/\delta \rfloor}
[\varepsilon + \ell \delta, \varepsilon + (\ell+1) \delta]$, we have by a simple 
union bound combined with \eqref{eq:d-bound1} that 
\begin{equation} \label{eq:d-bound}
\limsup_{n \to \infty} \p ( D_n \cap B_n | 
X_{A(n\varepsilon)} > 0) \leq C(K) \delta^2.
\end{equation}

By \eqref{eq:wbound-3} we can choose $K$ large enough to make the 
limit superior (in $n$) of the third term in \eqref{eq:w'bound} arbitrarily small.
Then with this fixed $K$ we can apply \eqref{eq:BC-bound} and \eqref{eq:d-bound}.
Summarizing, we obtain that \eqref{eq:bill2} holds, thus the proof is complete.
\end{proof}

\begin{proof}[Proof of Corollary \ref{cor:reverse}]
Since by Theorem \ref{thm:fdd>1} we have convergence on $D([\varepsilon, \infty))$ for arbitrary $\varepsilon > 0$ and the limit is a.s.~continuous at $t=1$, Theorem 16.2 in \cite{Billingsley} implies that the convergence also holds on $D([\varepsilon, 1])$.

Let $\psi_m: D([1, m]) \to D([1,m])$,
\[
(\psi_m x)(t) = \begin{cases}
x(t), & \text{if $1 \leq t\leq m-1$}, \\
(m-t) \, x(t), & \text{if $m-1\leq t \leq m$},
\end{cases}
\]
and define $\phi_\varepsilon: D([\varepsilon, 1]) \to D([1,\frac{1}{\varepsilon}])$ such that 
\[
(\phi_\varepsilon x)(t)  = x \left( \frac{1}{t} - \right),
\]
where $x(u - )$ stands for the left limit in $u$.
It can easily be proven that $\psi_m$ and $\phi_\varepsilon$ are continuous.
Thus, the statement follows from Lemma 3 of Section 16 in \cite{Billingsley} after applying the mapping theorem with $\psi_\frac{1}{\varepsilon} \circ \phi_\varepsilon$.
\end{proof}

\subsection{Proofs of Subsection \ref{subsec:BPVE-immig}}

Recall that $h_n(s) = \sum_{k=0}^\infty h_n[k] s^k$ denotes the g.f.~of the offspring distribution, $g_n(s) = \E s^{Y_n}$ is the g.f.~of the population size 
and let
\[
g_{j,n} (s) = \E[s^{Y_n} | Y_j = 0]
\]
denote the conditional g.f.~of the process $(Y_n)_{n\in\N}$.

The following result assures that the limiting process is an honest, standard 
continuous time Markov branching process with immigration. For general 
definitions on continuous time Markov chains we refer to Anderson \cite{Anderson}.

\begin{lemma}\label{lemma:funclimit-bev}
Assume that conditions (C1)--(C3) and either (C4) or (C4') hold.
Then for the BPVEI $(Y_n)$ defined in \eqref{eq:Y},
\[
\lim_{n \to \infty} g_{A(nu),A(nt)}(s) =
\frac{f_Y(s)}{f_Y(h_{u/t}(s))},
\]
and for any $y \in \N$
\[
\lim_{n \to \infty} \E [ s^{Y_{A(nt)}} | Y_{A(nu)} = y ]
= \frac{f_Y(s)}{f_Y(h_{u/t}(s))} \left( h_{u/t}(s) \right)^y.
\]
Furthermore, the transition generating functions
\begin{equation*} 
G_y(s, t) = \frac{f_Y(s)}{f_Y(h_{e^{-t}}(s))} \left( h_{e^{-t}}(s) \right)^y
\end{equation*}
define a homogeneous standard honest transition function of a Markov branching 
process with immigration.
\end{lemma}

\begin{proof}
By the branching property (see e.g.~the proof of Theorem 2 in \cite{KevKub}),
\[
g_{j,n}(s) = \prod_{l=j+1}^n h_l(f_{l,n}(s))
= \frac{g_n(s)}{g_j(f_{j,n}(s))}.
\]
Therefore, by Theorem 2 or 3 in \cite{KevKub} and Lemma \ref{lemma:funclimit1}
\[
\lim_{n \to \infty} g_{A(nu), A(nt)}(s) = 
\lim_{n \to \infty} \frac{g_{A(nt)}(s)}{g_{A(nu)}(f_{A(nu), A(nt)}(s))} = 
\frac{f_Y(s)}{f_Y(h_{u/t}(s))}.
\]

The second statement follows similarly, upon noting that 
\[
\E [ s^{Y_k} | Y_\ell = y] = g_{\ell, k}(s) (f_{\ell, k}(s))^{y}.
\]

Since $G_y(\cdot,t)$ is a limit of generating functions and $G_y(1,t) = 1$, it 
is clear that $G_y(\cdot, t)$ is generating function, thus it determines 
the transition functions via 
\[
G_y(s, t) = \sum_{j=0}^\infty p_{y,j}(t) s^j,
\]
where $p_{y,j}(t) \in [0,1]$, $\sum_{j=0}^\infty p_{y,j}(t) \equiv 1$, 
$p_{y,j} (0) = \delta_{y,j}$. Furthermore, it satisfies the Chapman--Kolmogorov
equation, since it is a limit of generating functions of discrete time 
Markov process. The branching property follows from the explicit form of 
$G_y(s,t)$; see \cite{LiChenPakes}.
\end{proof}

By Theorem 2.1 in \cite{LiChenPakes}, see the displayed equation after (2.2),
\[
a(s) = \lim_{t \downarrow 0} \frac{\partial}{\partial t} h_{e^{-t}}(s)
= (1-s) \left( 1 + \frac{\nu}{2} (1-s) \right),
\]
while 
\[
b(s) = \lim_{t \downarrow 0} \frac{\partial}{\partial t} 
\frac{f_Y(s)}{f_Y(h_{e^{-t}}(s))} = \sum_{k=1}^\kappa \lambda_k (s-1)^k,
\]
where the second equality follows after a straightforward calculation.
Recall that under (C4) $\kappa = K -1 $, and under (C4') $\kappa = \infty$.
From \eqref{eq:def-ab} we obtain that the offspring and immigration 
rates and generating functions are given by 
\[
\alpha = 1 + \nu, \quad 
f(s) = \frac{1+\frac{\nu}{2}}{1+\nu} + \frac{\frac{\nu}{2}}{1+\nu} s^2,
\]
and 
\[
\beta = - b(0), \quad h(s) = 1 + \frac{b(s)}{\beta}.
\]

\begin{proof}[Proof of Theorem \ref{thm:conv-immig}]
The proof is similar to the proof of Theorem \ref{thm:fdd>1}, so we only sketch 
it, emphasizing the differences.

The convergence of finite dimensional distributions follows from the 
convergence of one dimensional marginals, and the convergence of the 
transition probabilities. The first statement is Theorem 2 or 3 in \cite{KevKub}, 
the second is Lemma \ref{lemma:funclimit-bev}.

We prove tightness of $(Y_{A(nt)})_{t\geq \varepsilon}$ for all $\varepsilon > 0$ fixed.
Similarly to the proof of Theorem \ref{thm:fdd>1} it is enough to check that
\begin{equation}\label{eq:Y>K}
\lim_{K\to\infty}\limsup_{n\to\infty} 
\p(\max_{A(n\varepsilon) \leq k \leq A(nm)} Y_k > K) = 0    
\end{equation}
and
\begin{equation}\label{eq:cont-mod}
\lim_{\delta\to\infty} \limsup_{n\to\infty}\p(w'((Y_{A(nt)})_{t\in[\varepsilon, m]}, \delta) \geq \eta) = 0, 
\end{equation}
for all $\eta > 0$.
By Lemma \ref{lemma:anjk}
\[
\lim_{n\to\infty} \E(Y_{A(nm)}) = 
\lim_{n \to \infty} \sum_{j=1}^{A(nm)} 
\frac{m_{j,1}}{1 - \fl_j} (1 - \fl_j) \fl_{j,A(nm)} = \lambda_1 <\infty.
\]
Let $\tau = \min\{ k: \, A(n\varepsilon) \leq k \leq A(nm), Y_k > K \}$,
and with $C_0$ defined in \eqref{eq:C_0},
\[
\begin{split}
\E(Y_{A(nm)} \ind_{\{A(n\varepsilon) \leq \tau \leq A(nm)\}}) & 
\geq \sum_{k=A(n\varepsilon)}^{A(nm)} \left[
K \fl_{k, A(nm)} + \sum_{j=k+1}^{A(nm)} m_{j,1} \fl_{j,A(nm)}
\right] \p(\tau = k) \\
& \geq \frac{K}{C_0} \fl_{A(n\varepsilon), A(nm)} \p(\max_{A(n\varepsilon) \leq k \leq A(nm)} Y_k > K),
\end{split}
\]
thus \eqref{eq:Y>K} holds.
Again, to prove \eqref{eq:cont-mod} we need to bound the probability of 
having many jumps or at least three close jumps. 
As in \eqref{eq:aux-f-n-bound} and \eqref{eq:prob-jump}, for any $y_k \leq K$, with 
$h_{k+1}[0] = \p(\varepsilon_{k+1} = 0)$,
\[
\begin{split}
\p(Y_{k+1} \neq Y_k | Y_k = y_k) & \leq 1 - (f_{k+1}[1])^K h_{k+1}[0] 
\leq c K |1-\fl_{k+1}|,
\end{split}
\]
since by our assumptions on the immigration, 
\[
h_{k+1}[0] = h_{k+1}(0) \geq 1 - m_{k+1,1} \geq 1 - c |1-\fl_{k+1}|.
\]
Hence, for any $N$, $y_{k_i} \neq y_{k_i +1}$, $y_{k_i} \leq K$, $i=1,\ldots, N$
\[
\p(Y_{k_i} = y_{k_i}, Y_{k_i+1} = y_{k_i+1}, i=1,\dots, N)
\leq (cK)^N |1-\fl_{k_1}| \cdots |1-\fl_{k_N}|,
\]
and thus with $N=\lfloor \frac{m-\varepsilon}{\delta}\rfloor + 1$
\[
\begin{split}
& \p((Y_{A(nt)})_{t\in [\varepsilon, m]} \text{ has at least $N$ jumps}, \max_{A(n\varepsilon) \leq k \leq A(nm)} Y_{k} \leq K) \\
& \leq \sum_{A(n\varepsilon) \leq k_1 < \dots <k_N \leq A(nm)}
\p(Y_{k_i} \neq Y_{k_i+1}, i=1, \dots, N, \max_{A(n\varepsilon) \leq k \leq A(nm)} Y_k \leq K) \\
& \leq (cK^3)^N \frac{1}{N!}\left(
\sum_{k=A(n\varepsilon)}^{A(nm)} |1-\fl_k|
\right)^N \leq \frac{C(K)^N}{N!},
\end{split}
\]
and similarly for any fixed $t_0 \in [\varepsilon, m]$ and $\delta > 0$,
\[
\begin{split}
& \p((Y_A(nt))_{t\in [t_0, t_0 + \delta]} \text{has at least 3 jumps}, \max_{A(n\varepsilon) \leq k \leq A(nm)} Y_k \leq K) \\
& \leq ((cK + d)K(K-1))^3 \frac{1}{3!} \left(
\sum_{k=A(nt_0)}^{A(n(t_0 + \delta))} |1-\fl_k|
\right)^3 \leq C(K) \delta^3,
\end{split}
\]
thus \eqref{eq:cont-mod} follows.
\end{proof}

The proof of Corollary \ref{corr:imm} is the same as the proof of Corollary \ref{cor:reverse}.

\medskip

\noindent \textbf{Acknowledgement.}
This research was supported by the Ministry of Culture and Innovation of 
Hungary from the National Research, Development and Innovation Fund, project no.~TKP2021-NVA-09.


\end{document}